\documentclass[11pt]{article}

\usepackage{times}
\usepackage{amsthm}
\usepackage{amsmath}
\usepackage{amssymb}
\usepackage{mathrsfs}
\usepackage{pb-diagram}
\usepackage{color}

\usepackage[colorlinks]{hyperref}

\def\th{\theta}
\def\Th{\Theta}
\def\O{\Omega}
\def\N{\mathbb{N}}
\def\T{\mathbb{T}}

\def\X{{\sf X}}

\def\H{\mathcal H}

\def\Q{\mathcal Q}

\def\R{\mathbb{R}}

\def\e{{\sf e}}

\def\m{{\sf m}}

\def\({\left(}
\def\[{\left[}
\def\){\right)}
\def\]{\right]}

\def\p{\parallel}
\def\<{\langle}
\def\>{\rangle}

\def\supp{\mathop{\mathrm{supp}}\nolimits}

\newtheorem{Theorem}{Theorem}[section]
\newtheorem{Remark}[Theorem]{Remark}
\newtheorem{Lemma}[Theorem]{Lemma}

\newtheorem{Proposition}[Theorem]{Proposition}
\newtheorem{Definition}[Theorem]{Definition}
\newtheorem{Example}[Theorem]{Example}

\numberwithin{equation}{section}

\begin{document}


\title{Anisotropic Gohberg Lemmas for\\ Pseudodifferential Operators on Abelian Compact Groups}

\date{\today}

\author{M. M\u antoiu \footnote{
\newline
{
}}
}

\date{\small}

\maketitle \vspace{-1cm}


\begin{abstract}
Classically, Gohberg-type Lemmas provide lower bounds for the distance of suitable pseudodifferential operators acting in a Hilbert space to the ideal of compact operators, in terms of "the behavior of the symbol at infinity". In this article the pseudodifferential operators are associated to a compact Abelian group $\X$ and an important role is played by its Pontryagin dual $\widehat\X$\,. H\"ormander- type classes of symbols are not always available; they will be replaced by crossed product $C^*$-algebras involving a vanishing oscillation condition, which anyway is more general even in the particular cases allowing a full pseudodifferential calculus. In addition, the distance to a large class of operator ideals is controlled; the compact operators only form a particular case. This involves invariant closed subsets of certain compactifications of the dual group or, equivalently, invariant ideals of $\ell^\infty(\widehat\X)$\,.
\end{abstract}

\section{Introduction}\label{didirlish}

A compact Abelian Lie group $\X$ is given, with Pontryagin dual $\Xi$\,. This one is an Abelian discrete group (the most general one).
For suitable functions $f:\X\!\times\!\Xi\to\mathbb C$ one defines {\it the pseudodifferential operator} in $\H:=L^2(\X)$ 
\begin{equation}\label{zurzurel}
\begin{aligned}
\big[{\sf Op}(f)u\big](x):&=\int_\X\sum_{\xi\in\Xi}\xi\big(xy^{-1}\big)f(x,\xi)u(y)dy\\
&=\sum_{\xi\in\Xi}\xi(x)f(x,\xi)\widehat u(\xi)\,.
\end{aligned}
\end{equation}

Maybe changing the setting, typically, results \'a la Goldberg \cite{Go,Gr,Ho,KN,MW,DR,RVR} state that
\begin{equation}\label{ares}
\p\!{\sf Op}(f)-K\!\p_{\mathbb B(\H)}\,\ge\limsup_{\infty}|f(x,\xi)|
\end{equation}
for every compact operator $K$, under various circumstances upon the setting and the function $f$. Some of the references only treat singular integral operators. The recent paper \cite{RVR} refers to situations in which there is no group structure.  \cite{DR} applies to maybe non-commutative compact Lie groups; the function $f$ is then matrix-valued and this brings extra complications. 

\smallskip
The remarkable fact is that in \eqref{ares} the operator norm of a certain operator involving a non-commutative calculus with functions $f$ admits a lower bound in terms of the point-wise (asymptotic) behavior of $f$. We were deliberately imprecise about the type of $\limsup$ which is involved; this depends on the framework. In our case one has $\underset{\xi\to\infty}{\limsup}$\,, basically because the group $\X$ is compact. 

\smallskip
We are going to call {\it symbol} a function on the {\it phase space} $\X\!\times\!\Xi$ to which the quantization (= pseudodifferential calculus) ${\sf Op}$ applies. Only in limited situations it is supposed to belong to some H\"ormander classes $S^m_{\rho,\delta}(\X\!\times\!\Xi)$\,; in certain cases this could not even make sense. Partially Fourier transformed crossed product $C^*$-algebras supply here convenient symbol classes.
For the rather basic facts we use about crossed products we refer to \cite{Wi}.

\smallskip
One of the purposes of the present article is to treat situations when symbol classes of the type $S^m_{\rho,\delta}(\X\!\times\!\Xi)$ are not defined, covering all the Abelian compact groups (unfortunately, the non-Abelian groups seem to be out of reach by the present methods). In addition, even when they do exist, they are only strictly included in our classes, which involve a vanishing oscillation condition in the variable $\xi\in\Xi$\,. This condition is also simpler and it looks optimal for the selected setting.

\smallskip
Of course, in \eqref{ares} the left hand side could be replaced by $\,\inf_{K}\!\p\!{\sf Op}(f)-K\!\p_{\mathbb B(\H)}$\,, the infimum being taken over all the compact operators. But this expression coincides with the norm of the image of ${\sf Op}(f)$ in the quotient Calkin algebra. This hints to the fact that ${\sf Op}(f)$ should belong to a full $C^*$-algebra and not only to families defined by (some generalized form of) smoothness and decay. It turns out that this one is, as mentioned above, the partial Fourier transform of the crossed product of an Abelian $C^*$-algebra of vanishing oscillation functions on the dual group $\Xi$ by the action of $\Xi$ by translations. See \cite{M} and references cited therein for applications of such crossed products in spectral analysis and \cite{Ma} for various generalizations of vanishing oscillation in a different context. We also refer to contributions of N. Higson and J. Roe, with many consequences in Noncomutative Geometry. We indicate once again that this condition appears in a very natural way in the proof in Section \ref{bergman}. 

\smallskip
But another interesting question is raised by the interpretation above, which is treated in the present paper. One could ask for lower bounds for the distance from a given pseudodifferential operator to a suitable ideal of operators, the compact ones only being a particular case. The framework and the techniques of proof points towards ideals of symbols constructed, also through a partial Fourier transform, from crossed products with the group $\Xi$ acting on invariant ideals of $\ell^\infty(\Xi)$\,. These ones are in a one-to-one correspondence with closed invariant subsets $\O$ of the corona space $\delta\Xi:=\beta\Xi\!\setminus\!\Xi$\,, where the Stone-\u Cech compactification $\beta\Xi$ is the Gelfand spectrum of $\ell^\infty(\Xi)$\,. Besides defining such an ideal, $\O$ also defines a sort of generalization of {\it the vanishing oscillation property} (also called {\it slowly oscillation}), associated with an optimal class of symbols ${\sf VO}^\O(\X\!\times\!\Xi)$ for which an extension of Gohberg's Lemma holds. The lower bound is in this case ${\sf D^\O(f)}=\underset{\xi\to\O}{\limsup}\,\underset{x\in\X}{\sup}|f(x,\xi)|$\,. The presence of $\O$ is responsible for the term "anisotropic" in the title. The standard case corresponds to the choice $\O=\delta\Xi$\,. 

\smallskip
In Section \ref{bogart} we describe the framework and make the necessary preparations. The main result is stated and proved in Section \ref{bergman}. A final section is dedicated to some complementary remarks.

\smallskip
{\bf Conventions:} For every Hilbert space $\H$ we denote by $\mathbb  B(\H)$ the $C^*$-algebra of all bounded linear operators and by $\mathbb K(\H)$ the ideal of all compact operators. Homomorphisms and, in particular, representations of $C^*$-algebras are supposed to commute with the involutions. Ideals are implicitly assumed to be bi-sided, closed and self-adjoint. If $Y$ is a locally compact Hausdorff space, $C(Y)$ denotes all the continuous complex functions and $C_0(Y)$ the functions which, in addition, converge to zero at infinity. If $Y$ is discrete, we use consacrated notation as $c_0(Y)$ and $\ell^p(Y)$\,.

\section{Algebras of pseudifferential operators on Abelian compact groups}\label{bogart}

Let $\X$ be a compact Abelian group, with unit $\e$ and with (Abelian discrete) Pontryagin dual $\widehat\X\equiv \Xi$\,. Both group laws will be denoted multiplicatively. The duality will be denoted by both expressions $\xi(x)$ and $x(\xi)$\,, which is allowed by the Pontryagin-van Kampen Theorem. On $\X$ we consider the normalized Haar measure $d\m(x)\equiv dx$\,. {\it The Fourier transformation} $\mathcal F:L^1(\X)\to c_0(\Xi)$ can also be defined as a unitary operator $\mathcal F:L^2(\X)\to \ell^2(\Xi)$ and it is essentially given by 
\begin{equation*}\label{ayshlan}
(\mathcal F\varphi)(\xi):=\int_\X\overline{\xi(x)}\varphi(x)dx\,,
\end{equation*}
with inverse
\begin{equation*}\label{alyiar}
\big(\mathcal F^{-1}\psi\big)(x):=\sum_{\xi\in\Xi}\xi(x)\psi(x)\,.
\end{equation*}
The Riemann-Lebesgue Lemma also insures the interpretation $\mathcal F^{-1}\!:\ell^1(\Xi)\to C(\X)$ (linear contraction).

\smallskip
Recall the isomorphism $\ell^\infty(\Xi)\cong C(\beta\Xi)$\,, where $\beta\Xi=\Xi\sqcup\delta\Xi$ is {\it the Stone-\u Cech compactification} of $\Xi$\,. We denote by $\th:\Xi\to{\rm Homeo}(\beta\Xi)$ the canonical  extension to $\beta\Xi$ of the action of $\Xi$ on itself by translations. It lifts to an action on $C(\beta\Xi)$ by 
\begin{equation*}\label{ertugrul}
\big[\Th_\eta(\psi)\big](\xi):=\psi\big[\th_{\eta}(\xi)\big]\,,\quad\forall\,\xi\in\beta\Xi\,,\,\eta\in\Xi\,.
\end{equation*}
With this action one can form {\it the crossed product} \cite{Wi} $\mathfrak C:=\Xi\ltimes\ell^\infty(\Xi)$\,.  It is the enveloping $C^*$-algebra of the Banach $^*$-algebra $\ell^1\big(\Xi;\ell^\infty(\Xi)\big)$ (isomorphic as a Banach space to the projective tensor product $\ell^1(\Xi)\overline\otimes\ell^\infty(\Xi)$)\,, with the obvious norm, the composition law
\begin{equation*}\label{kukuruja}
(\Phi\diamond\Psi)(\xi):=\sum_{\eta\in\Xi}\Th_\eta\big[\Phi(\xi\eta^{-1})\big]\Psi(\eta)
\end{equation*}
and the involution
\begin{equation*}\label{ibn}
\Psi^\diamond(\xi):=\Th_\xi\big[\Psi(\xi^{-1})^*\big]\,.
\end{equation*}
Using notations as $\big[\Psi(\xi)\big](\zeta)\equiv\Psi(\xi,\zeta)$\,, the algebraic structure of $\ell^1\big(\Xi;\ell^\infty(\Xi)\big)$ can be reformulated as 
\begin{equation*}\label{kulturija}
(\Phi\diamond\Psi)(\xi,\zeta):=\sum_{\eta\in\Xi}\Phi\big(\xi\eta^{-1}\!,\zeta\eta\big)\Psi(\eta,\zeta)\,,
\end{equation*}
\begin{equation*}\label{ben}
\Psi^\diamond(\xi,\zeta):=\overline{\Psi\big(\xi^{-1}\!,\zeta\xi\big)}\,.
\end{equation*}
There is a natural {\it (Schr\"odinger) representation} of $\mathfrak C$ in the Hilbert space $\ell^2(\Xi)$ given for $\Psi\in\ell^1\big(\Xi;\ell^\infty(\Xi)\big)$ by
\begin{equation}\label{arabi}
\big[{\sf sch}(\Psi)w\big](\xi):=\sum_{\eta\in\Xi}\Psi\big(\xi\eta^{-1}\!,\eta\big)w(\eta)\,.
\end{equation}
It is clear that ${\sf sch}$ also transforms isomorphically the Hilbert space $\ell^2(\Xi\!\times\!\Xi)$ into the Hilbert space $\mathbb B^2\big[\ell^2(\Xi)\big]$ of all Hilbert-Schmidt operators in $\ell^2(\Xi)$\,.

\smallskip
The operator \eqref{zurzurel} is an integral operator with kernel $\kappa_f(x,y):=\big(\mathbb F^{-1}f\big)\big(x,xy^{-1}\big)$\,, where $\mathbb F^{-1}\!:={\rm id\otimes\mathcal F^{-1}}$ is the inverse Fourier transformation in the second variable. It is a standard simple fact that ${\sf Op}$ is a Hilbert space isomorphism between $L^2(\X\!\times\!\Xi)\cong L^2(\X)\otimes\ell^2(\Xi)$ and the space $\mathbb B^2\big[L^2(\X)\big]$ of all Hilbert-Schmidt operators in $L^2(\X)$\,. (Below we will indicate another, $C^*$-algebraic meaning.)
We note that ${\sf Op}(\varphi\otimes\psi)=\varphi(Q)\psi(P)$\,, in terms of a multiplication operator $\varphi(Q)$ and the convolution operator
\begin{equation*}\label{gangurel}
\psi(P)=\mathcal F^{-1}\psi\big(Q_\Xi\big)\mathcal F={\sf Conv}\big(\mathcal F^{-1}\psi\,\big)\,.
\end{equation*}
Here $\psi\big(Q_\Xi\big)$ denoted the operator of multiplication by $\psi$ in $\ell^2(\Xi)$\,.

\smallskip
We make now the connection between ${\sf sch}$ and ${\sf Op}$\,. Defining the operator ${\sf op}(f)$ on $\ell^2(\Xi)$ by the unitary equivalence
\begin{equation*}\label{bey}
{\sf op}(f):=\mathcal F\,{\sf Op}(f)\mathcal F^{-1}\!\equiv\mathbf U^{-1}_\mathcal F\big[{\sf Op}(f)\big]\,,
\end{equation*} 
and applying the Fourier Inversion Formula, one gets explicitly
\begin{equation}\label{cavdar}
\big[{\sf op}(f)w\big](\xi)=\sum_{\eta\in\Xi}\int_\X x\big(\xi^{-1}\eta\big)f(x,\eta)w(\eta)dx=\sum_{\eta\in\Xi}\int_\X x\big(\xi\eta^{-1}\big)f\big(x^{-1}\!,\eta\big)w(\eta)dx\,.
\end{equation}
It is a pseudodifferential operator ${\sf Op}_\Xi^R(f\circ\mu)$ on the discrete Abelian group $\Xi$\,, with respect to a natural {\it right quantization}, for the symbol 
\begin{equation}\label{bazar}
\Xi\times\X\ni(\eta,x)\to f\big[\mu(\eta,x)\big]:=f\big(x^{-1}\!,\eta\big)\in\mathbb C\,.
\end{equation} 
We mention in passing that a full notation for ${\sf Op}(f)$ could be ${\sf Op}_\X^L(f)$\,, to mean that a {\it left quantization} (Kohn-Nirenberg type) is considered, starting with the group $\X$\,. In terms of the Schr\"odinger representation \eqref{arabi} of the crossed product, ${\sf op}$ satisfies
\begin{equation*}\label{kopek}
{\sf op}\big[\big(\mathcal F^{-1}\!\otimes{\rm id}\big)(\Psi)\big]={\sf sch}(\Psi)\,.
\end{equation*}
The final conclusion is that
\begin{equation*}\label{didirlis}
{\sf Op}=\mathbf U_\mathcal F\circ{\sf sch}\circ(\mathcal F\otimes{\rm id})\,,
\end{equation*}
which can be summarized in the diagram
\begin{equation*}\label{diagram}
\begin{diagram}
\node{\ell^2(\Xi\!\times\!\Xi)}\arrow{e,t}{\mathcal F^{-1}\!\otimes{\rm id}} \arrow{s,l}{{\sf sch}}\node{L^2(\X\!\times\!\Xi)}\arrow{s,r}{{\sf Op}}\arrow{sw,t}{{\sf op}}\\ 
\node{\mathbb B^2[\ell^2(\Xi)]} \arrow{e,t}{\mathbf U_\mathcal F} \node{\mathbb B^2[L^2(\X)]}
\end{diagram}
\end{equation*}
composed of Hilbert space isomorphisms, or the diagram
\begin{equation}\label{diagramm}
\begin{diagram}
\node{\mathfrak C}\arrow{e,t}{\mathcal F^{-1}\!\otimes{\rm id}} \arrow{s,l}{{\sf sch}}\node{\mathfrak D}\arrow{s,r}{{\sf Op}}\arrow{sw,t}{{\sf op}}\\ 
\node{\mathbb B[\ell^2(\Xi)]} \arrow{e,t}{\mathbf U_\mathcal F} \node{\mathbb B[L^2(\X)]}
\end{diagram}
\end{equation}
in which ${\sf sch},{\sf op},{\sf Op}$ are representations of $C^*$-algebras, the last two being unitarily equivalent. 
\smallskip
We introduced the partial Fourier transform $C^*$-algebra 
$$
\mathfrak D:=\big(\mathcal F^{-1}\!\otimes\!{\rm id}\big)\mathfrak C\equiv\mathbf F^{-1}\mathfrak C
$$ 
with the transported algebraical and topological structure.  
Using the Riemman-Lebesgue Lemma and the universal property of the crossed product, it follows that 
\begin{equation}\label{ana}
\mathfrak D\subset C\big(\X;\ell^\infty(\Xi)\big)\cong C(\X\!\times\!\beta\Xi)\,.
\end{equation}
Of course, the norms on $\mathfrak D$ and $C(\X\!\times\!\beta\Xi)$ are different. Below, obvious versions of the diagram \eqref{diagramm} will involve restrictions to $C^*$-subalgebras of $\mathfrak C$\,.

\smallskip
We consider a $\th$-invariant closed (thus compact) subset $\O$ of $\delta\Xi$ and set 
\begin{equation*}\label{suleyman}
\ell^{\infty,\O}(\Xi)\cong C^\Omega(\beta\Xi):=\big\{\tilde\psi\in C(\beta\Xi)\,\big\vert\,\tilde\psi|_\O=0\big\}\,.
\end{equation*}
It is a (bi-sided, closed, self-adjoint) $\Th$-invariant ideal of $C(\beta\Xi)$\,. Here $\tilde\psi\equiv\beta\psi$ should be seen as the extension by continuity of $\psi\in\ell^\infty(\Xi)$\,. The quotient $\ell^{\infty}(\Xi)/\ell^{\infty,\O}(\Xi)$ is isomorphic to $C(\O)$\,. It is often useful to express simple objects and properties connected to $\tilde\psi$ in terms of its restriction $\psi$ to "the finite part" $\Xi$\,. We define $\mathscr V(\O)$ to be the filter of all neighborhoods of $\Omega$ in $\beta\Xi$ and 
\begin{equation*}\label{gudurel}
{\sf d}^\O(\psi)\equiv\limsup_{\xi\to\O}|\psi(\xi)|:=\!\underset{V\in\mathscr V(\O)}{\lim}\sup_{\xi\in V\cap\Xi}|\psi(\xi)|\,.
\end{equation*}
(One could also work only with a basis of neighborhoods.)

\begin{Lemma}\label{gundogdu}
One has ${\sf d}^\O(\psi)=\,\big\Vert\,\tilde\psi|_\O\big\Vert_\infty$ and, consequently, 
\begin{equation*}\label{selcan}
\ell^{\infty,\O}(\Xi)=\Big\{\psi\in\ell^\infty(\Xi)\,\Big\vert\,\limsup_{\xi\to\O}|\psi(\xi)|=0\Big\}\,.
\end{equation*}
There exists a net $(\xi_i)_{i\in I}$ in $\Xi$ such that $\xi_i\!\to\O$ and $|\psi\big(\xi_i\big)|\to{\sf d}^\O(\psi)$\,.
\end{Lemma}

\begin{proof}
Simple well-known facts in topology. We recall that $\xi_i\!\to\O$ means that for each $V\in\mathscr V(\O)$ there is some $i_V\in I$ with $\xi_i\in V$ if $i\ge i_V$\,.
\end{proof}

We also have the ideals 
\begin{equation*}\label{anatolia}
\mathfrak C^\O\!:=\Xi\!\ltimes\!\ell^{\infty,\O}(\Xi)\subset\mathfrak C\quad{\rm and}\quad\mathfrak D^\O\!:=\big(\mathcal F^{-1}\!\otimes\!{\rm id}\big)\mathfrak C^\O\equiv\mathbf F^{-1}\mathfrak C^\O\subset\mathfrak D\,.
\end{equation*} 
Note that $\mathfrak D^\emptyset\!=\mathfrak D$ and $\mathfrak D^{\delta\Xi}=\mathbf F^{-1}\big(\Xi\!\ltimes\!c_0(\Xi)\big)$\,. One has $X\!\times\!\beta\Xi=(X\!\times\!\Xi)\sqcup(X\!\times\!\delta\Xi)$\,. Using the filter $\{\X\!\times\! V\!\mid\!V\in\mathscr V(\O)\}$ of all the neighborhoods of $\X\!\times\!\O$ in $\X\!\times\!\beta\Xi$\,, an analog of Lemma \ref{gundogdu} is

\begin{Lemma}\label{gumustekin}
Let $f\in\mathfrak D$\,, with continuous extension $F$ to $\X\!\times\!\beta\Xi$ (by \eqref{ana})\,. One has 
\begin{equation}\label{gugustiuc}
{\sf D}^\O(f)\equiv\limsup_{(x,\xi)\to\X\times\O}|f(x,\xi)|=\,\p\!F|_{\X\times\O}\Vert_\infty
\end{equation} 
and 
\begin{equation*}\label{selchan}
\mathfrak D^\O=\Big\{f\in\mathfrak D\,\Big\vert\,\limsup_{(x,\xi)\to\X\times\O}|f(x,\xi)|=0\Big\}\,.
\end{equation*}
There exists a point $x_0\in\X$ and a net $\Xi\ni\xi_i\!\to\O$ such that $|f\big(x_0,\xi_i\big)|\to{\sf D}^\O(f)$\,.
\end{Lemma}

\begin{proof}
The first statement is obvious. For the last statement, we choose a point $(x_0,\xi_0)$ of the compact set $\X\!\times\!\O$ for which $\p\!F|_{\X\times\O}\Vert_\infty=F\big(x_0,\xi_0\big)$\,, we approximate (by density) $\xi_0$ by elements of $\Xi$ and then use the continuity of $F$ as well as \eqref{gugustiuc}.
\end{proof}

\begin{Remark}\label{cicek}
{\rm Let $f\in\mathfrak D^\O$. Since $\O$ is supposed to be invariant, for every $\eta\in\Xi$ one has 
\begin{equation*}\label{aykiz}
\limsup_{(x,\xi)\to\X\times\O}|f(x,\eta\xi)|=0\,.
\end{equation*}
Of course, this is the same as
\begin{equation*}\label{naykiz}
\limsup_{\xi\to\O}\sup_{x\in\X}|f(x,\eta\xi)|=0\,.
\end{equation*}
}
\end{Remark}

We are going to introduce now the class of symbols for which our main result holds.

\begin{Definition}\label{bitinia}
\begin{enumerate}
\item[(a)]
For  functions $\psi:\Xi\to\mathbb C$\,, $f:\X\!\times\!\Xi\to\mathbb C$ and elements $\zeta\in\Xi$ we set
\begin{equation*}\label{alp}
{\sf osc}^\zeta_\psi(\xi):=\psi(\xi\zeta)-\psi(\xi)\,,
\end{equation*}
\begin{equation*}\label{kalp}
{\sf OSC}^\zeta_f(x,\xi):=f(x,\xi\zeta)-f(x,\xi)\,.
\end{equation*}
\item[(b)]
One defines the sets of "vanishing oscillation functions in the direction $\O$"
\begin{equation*}\label{konya}
{\sf vo}^\O(\Xi):=\big\{\psi\in\ell^\infty(\Xi)\,\big\vert\,{\sf osc}^\zeta_\psi\!\in\ell^{\infty,\O}(\Xi)\,,\,\forall\,\zeta\in\Xi\big\}\,,
\end{equation*}
\begin{equation*}\label{erzurum}
{\sf VO}^\O(\X\!\times\!\Xi):=\big\{f\in\mathfrak D\,\big\vert\,{\sf OSC}^\zeta_f\!\in\mathfrak D^\O,\,\forall\,\zeta\in\Xi\big\}\,.
\end{equation*}
\end{enumerate}
\end{Definition}

\begin{Remark}\label{aktogali}
{\rm The map $\ell^\infty(\Xi)\ni\psi\to{\sf osc}^\zeta_\psi\in\ell^\infty(\Xi)$ is linear, it satisfies $\overline{{\sf osc}^\zeta_\psi}={\sf osc}^\zeta_{\overline\psi}$\;,
\begin{equation*}\label{vasilius}
{\sf osc}^\zeta_{\psi_1\psi_2}\!=\psi_2(\cdot\zeta){\sf osc}^\zeta_{\psi_1}\!+\psi_1{\sf osc}^\zeta_{\psi_2}\,,
\end{equation*}
\begin{equation*}\label{constantin}
{\sf osc}^\zeta_{\theta_\eta\psi}\!=\theta_\eta{\sf osc}^\zeta_\psi\,,\quad\forall\,\eta\in\Xi
\end{equation*}
and $\big\Vert{\sf osc}^\zeta_{\psi}\big\Vert_\infty\!\le 2\big\Vert\psi\big\Vert_\infty$\,.
Using this and the fact that $\ell^{\infty,\O}(\Xi)$ is an ideal in $\ell^{\infty}(\Xi)$ which is invariant under $\theta$-translations, it follows easily that $\ell^{\infty,\O}(\Xi)$ is an ideal in the $C^*$-subalgebra ${\sf vo}^\O(\Xi)$ of $\ell^\infty(\Xi)$\,. This $C^*$-algebra is unital and invariant under translations. It can also be defined by the condition
\begin{equation*}\label{niceea}
\lim_{\xi\to\O}|{\sf osc}|^M_\psi(\xi):=\lim_{\xi\to\O}\max_{\zeta\in M}|\psi(\xi\zeta)-\psi(\xi)|=0\,,\quad\forall\,M\subset\Xi\,,\,\#M<\infty\,.
\end{equation*}
This is why we name the operation "oscillation"; for $\#M=1$ this also deserves the name "difference operation".}
\end{Remark}

\begin{Remark}\label{aktogaly}
{\rm Similar statements hold for ${\sf OSC}^\zeta$. It is defined a priori on $C(\Xi)\!\otimes\ell^\infty(\Xi)\cong C(\X\!\times\!\beta\Xi)$\,, so it can be restricted to $\mathfrak D$. One shows that
\begin{equation*}\label{batuhan}
{\sf VO}^\O(\X\!\times\!\Xi)=\big(\mathcal F^{-1}\!\otimes{\rm id}\big)\big[\Xi\!\ltimes\!{\sf vo}^\O(\Xi)\big]
\end{equation*}
and that $\mathfrak D^\O$ is an ideal in this $C^*$-algebra (since the crossed product functor is exact \cite{Wi}, but this also follows from the initial definition). The particular case ${\sf VO}^\emptyset(\X\!\times\!\Xi)=\mathfrak D$ is only interesting as a sort of "universe", on which ${\sf Op}$ surely applies. Since ${\sf VO}^\O(\X\!\times\!\Xi)$ is the family of symbols for which our result applies, we make its defining condition more explicit:
\begin{equation*}\label{hachaturian}
\limsup_{(x,\xi)\to\X\times\O}|f(x,\xi\zeta)-f(x,\xi)|=0\,,\quad\forall\,\zeta\in\Xi\,.
\end{equation*}}
\end{Remark}

\section{Gohberg's Lemma}\label{bergman}

We now define the operator (represented) versions of the symbol classes introduced above. One sets 
\begin{equation}\label{sah}
\mathbb D^\O\equiv\big\{C(\X)\cdot \ell^{\infty,\O}(\Xi)\big\}:={\sf Op}\big(\mathfrak D^\O\big)\subset\mathbb B(\H)\,.
\end{equation}
It is an ideal in $\mathbb D^\emptyset\!=:\mathbb D\equiv\big\{C(\X)\cdot \ell^{\infty}(\Xi)\big\}:={\sf Op}\big(\mathfrak D\big)$ (but not in $\mathbb B(\H)$\,, of course). The reason of the notation in \eqref{sah} is that $\mathbb D^\O$ is generated by products of the form $\varphi(Q)\psi(P)$\,, with $\varphi\in C(\X)$ and $\psi\in\ell^{\infty,\O}(\Xi)$ (so $\mathbb D^\O$ is just the closure in $\mathbb B(\H)$ of the subspace generated by the products). From the diagram \eqref{diagramm} we  see that $\mathbb D^\O$ is unitarily equivalent with the ideal ${\sf sch}\big(\mathfrak C^\O\big)$ of the $C^*$-algebra ${\sf sch}\big(\mathfrak C\big)\subset\mathbb B\big[\ell^2(\Xi)\big]$\,. 

\smallskip
Similarly, one introduces 
\begin{equation*}\label{padisah}
\mathbb{VO}^\O\equiv\big\{C(\X)\cdot{\sf vo}^{\O}(\Xi)\big\}:={\sf Op}\big[{\sf VO}^\O(\X\!\times\!\Xi)\big]\subset\mathbb B(\H)\,.
\end{equation*}
In terms of crossed products, one can write
\begin{equation*}\label{padimat}
\mathbb{VO}^\O=\mathbb U_\mathcal F\big[{\sf sch}\big(\Xi\!\ltimes\!{\sf vo}^\O(\Xi)\big)\big]\,.
\end{equation*}
The pair formed of the $C^*$-algebra $\mathbb{VO}^\O$ and its ideal $\mathbb D^\O$ will play a central role. Note that the quotient $\mathbb{VO}^\O/\mathbb D^\O$ identifies to a $C^*$-subalgebra of $\mathbb D/\mathbb D^\O$. The canonical quotient epimorphism $\pi_\O:\mathbb D\to\mathbb D_\O:=\mathbb D/\mathbb D^\O$ restricts precisely to the canonical quotient epimorphism $\mathbb{VO}^\O\to\mathbb{VO}^\O/\mathbb D^\O$.

\begin{Theorem}\label{titus}
Let $f\in{\sf VO}^\O(\X\!\times\!\Xi)$\,. Then 
\begin{equation*}\label{petruccio}
\big\Vert\,\pi_\O\big[{\sf Op}(f)\big]\,\big\Vert_{\mathbb D_\O}\!={\rm dist}\big({\sf Op}(f),\mathbb D^\Omega\big)\ge {\sf D}^\O(f)\,.
\end{equation*}
\end{Theorem}

The line of the proof is inspired by \cite{DR}. We need first some preliminary results. 

\begin{Proposition}\label{gokce}
For $u\in\H=L^2(\X)$ and $i\in I$ (directed set) let us define 
\begin{equation}\label{haymana}
u_i(x)\equiv\big[{\sf E}(x_0,\xi_i)u\big](x):=\xi_i(x)u\big(xx_0^{-1}\big)\,.
\end{equation} 
We assume that $x_0\in\X$ and $\xi_i\to\O$\,. If $\,{\sf L}\in\mathbb D^\Omega$, then ${\sf L}u_i\to 0$ in norm.
\end{Proposition}

\begin{proof}
Since ${\sf L}$ can be approximated in norm by sums of products $\varphi(Q)\psi(P)$\,, with $\varphi\in C(\X)$ and $\psi\in\ell^{\infty,\O}(\Xi)$\,, it is enough to show that $\psi(P){\sf E}(x_0,\xi_i)=\mathcal F^{-1}\psi\big(Q_\Xi\big)\mathcal F\,{\sf E}(x_0,\xi_i)$ converges strongly to zero.
A simple computation shows that (operators in $\ell^2(\Xi)$)
\begin{equation*}\label{turgut}
\psi\big(Q_\Xi\big)\mathcal F\,{\sf E}(x_0,\xi_i)=\xi_i(x_0)x_0^{-1}(\Q_\Xi)\psi(Q_\Xi)T_{\xi_i^{-1}}\mathcal F\,,
\end{equation*}
where $x_0(Q_\Xi)$ is the (unitary) operator of multiplication with the function $\xi\to\xi\big(x_0^{-1}\big)\equiv x_0^{-1}(\xi)$ and $T_{\xi_i^{-1}}$ represents the translation by $\xi^{-1}_i$. So it is enough to show that $\psi(Q_\Xi)T_{\xi_i^{-1}}\!=T_{\xi_i^{-1}}\psi\big(\xi_i Q_\Xi)$ converges strongly to zero\,. By Remark \ref{cicek}, for every $\xi\in\Xi$ we have $\psi\big(\xi_i\xi\big)\to 0$\,, which implies the strong convergence to zero of the multiplication operator $\psi\big(\xi_i Q_\Xi\big)$ and the proof is finished.
\end{proof}

The next result makes an asymptotic connection between the pseudodifferential operator and a multiplication operator defined by its symbol. Its proof will put into evidence the oscillation of this symbol.

\begin{Lemma}\label{wildemir}
For $f\in{\sf VO}^\O(\X\!\times\!\Xi)$\,, $u\in\H$ and $u_i$ given by \eqref{haymana} one has
\begin{equation}\label{bamse}
\big\Vert{\sf Op}(f)u_i-f\big(\cdot,\xi_i\big)u_i\big\Vert\to 0\,.
\end{equation}
\end{Lemma}

\begin{proof} 
An easy argument shows that it is enough to prove \eqref{bamse} for $u$ belonging to a total subset $\H_0$ of $\H$ (meaning that the subspace generated by $\H_0$ is dense in $\H$)\,.
Such a set is formed of the elements $u^{\zeta}:=\mathcal F^{-1}\delta_\zeta$\,, with $\zeta\in\Xi$\,, for which 
\begin{equation*}\label{segurtekin}
u^\zeta_i(x)=\xi_i(x)\zeta\big(xx_0^{-1}\big)\,,\quad\forall\,\zeta\in\Xi\,,\,i\in I.
\end{equation*} 
By a direct computation
\begin{equation*}\label{tuktegin}
\begin{aligned}
\big[\big({\sf Op}(f)-f(\cdot,\xi_i)\big)u^\zeta_i\big](x)&=\xi_i(x)\zeta\big(xx_0^{-1}\big)\big[f(x,\xi_i\zeta)-f(x,\xi_i)\big]\\
&=\xi_i(x)\zeta\big(xx_0^{-1}\big){\rm OSC}_f^\zeta(x,\xi_i)
\end{aligned}
\end{equation*}
and it follows that
\begin{equation*}\label{ural}
\begin{aligned}
\big\Vert\big({\sf Op}(f)-f(\cdot,\xi_i)\big)u^\zeta_i\big\Vert^2&=\int_\X\big\vert\xi_i(x)\zeta\big(xx_0^{-1}\big){\rm OSC}_f^\zeta(x,\xi_i)\big\vert^2 dx\\
&=\int_\X\big\vert{\rm OSC}_f^\zeta(x,\xi_i)\big\vert^2 dx\\
&\le\sup_{x\in\X}\vert{\rm OSC}_f^\zeta(x,\xi_i)\big\vert^2\!\underset{n\to\infty}{\longrightarrow}0\,,
\end{aligned}
\end{equation*}
since $\m(\X)=1$ and $\xi_i\to\O$\,.
\end{proof}

\begin{proof} {\it of Theorem \ref{titus}.}
One must show that 
\begin{equation*}\label{cardash}
\p\!{\sf Op}(f)-{\sf L}\!\p\,\ge{\sf D}^\O(f)\quad{\rm for\ every}\ {\sf L}\in\mathbb D^\Omega.
\end{equation*}

\smallskip
Let $\big(x_0,\xi_i\big)\to\X\!\times\!\O$ such that $|f\big(x_0,\xi_i\big)|\to{\sf D}^\O(f)$ (by Lemma \ref{gumustekin}) and, for $u\in\H\!\setminus\!\{0\}$\,, define $u_i$ by \eqref{haymana}. Note that $\p\!u_i\!\p\,=\,\p\!u\!\p$\,. We deduce from Lemma \ref{wildemir} that for every $\epsilon>0$ there is $N\in\N$ such that for every $n\ge N$
\begin{equation}\label{tramse}
\big\Vert f\big(\cdot,\xi_i\big)u_i\big\Vert-\big\Vert{\sf Op}(f)u_i\big\Vert\le\epsilon\p\!u\!\p.
\end{equation}
Since  by \eqref{ana} $f$ is continuous in $x$ uniformly in $\xi$\,, there is a neighborhood $U$ of $x_0$ such that 
\begin{equation}\label{abdurrahman}
|f(x,\xi_i)-f(x_0,\xi_i)|\le\epsilon\,,\quad\forall\,x\in U,\,i\in I.
\end{equation}
Assuming now that $u\ne 0$ and $\supp u\subset Ux_0^{-1}$, i.\,e. $\supp u_i\subset U$, it follows from \eqref{abdurrahman} that
\begin{equation}\label{goncagul}
\big\vert f(x_0,\xi_i)\big\vert\p\!u_i\!\p-\big\Vert f\big(\cdot,\xi_i\big)u_i\big\Vert\le\epsilon\p\!u\!\p\,,\quad\forall\,i\in I.
\end{equation}
As a consequence of Proposition \ref{gokce} and \eqref{tramse}, \eqref{goncagul}, for $i$ big enough
$$
\begin{aligned}
\p\!{\sf Op}(f)-{\sf L}\!\p\p\!u\!\p&=\,\p\!{\sf Op}(f)-{\sf L}\!\p \p\!u_i\!\p\\
&\ge\,\p\!\big({\sf Op}(f)-{\sf L}\big)u_i\!\p\\
&\ge\,\p\!{\sf Op}(f)u_i\!\p-\p\!{\sf L}u_i\!\p\\
&\ge\,|f(x_0,\xi_i)|\p\!u\!\p-3\epsilon\p\!u\!\p\,,
\end{aligned}
$$
implying that
\begin{equation}\label{hamza}
\p\!{\sf Op}(f)-{\sf L}\!\p\,\ge\,|f(x_0,\xi_i)|-3\epsilon\underset{n\to\infty}{\longrightarrow}{\sf D}^\O(f)-3\epsilon\,.
\end{equation}
The proof is finished, since $\epsilon$ is arbitrary.
\end{proof}


\section{Some final remarks}\label{hafsa}

\begin{Remark}\label{actogali}
{\rm We obtained in \eqref{cavdar} and \eqref{bazar} the relationship ${\sf Op}_\Xi^R(g)={\sf op}\big(g\circ\mu^{-1}\big)$\,, with
\begin{equation*}\label{turkey}
\mu^{-1}:\X\!\times\!\Xi\to\Xi\!\times\!\X\,,\quad\mu^{-1}(x,\xi):=\big(\xi,x^{-1}\big)
\end{equation*}
for the right quantization of $g$ on the discrete Abelian group $\Xi$\,. Thus one has in $\ell^2(\Xi)$
\begin{equation*}\label{dundurel}
{\sf Op}_\Xi^R(g)=\mathcal F\,{\sf Op}\big(g\circ\mu^{-1}\big)\mathcal F^{-1}.
\end{equation*} 
On the other hand, by diagram \eqref{diagramm} (suitably restricted), 
\begin{equation*}\label{turcoman}
\mathcal F\,\mathbb D^\Omega\mathcal F^{-1}\!={\sf sch}\big(\Xi\!\ltimes\!\ell^{\infty,\O}(\Xi)\big)\quad{\rm and}\quad\mathcal F\,\mathbb{VO}^\Omega\mathcal F^{-1}\!={\sf sch}\big(\Xi\!\ltimes\!{\sf vo}^\O(\Xi)\big)\,.
\end{equation*}
Therefore we obtain from Theorem \ref{titus} the result
\begin{equation*}\label{petrucio}
{\rm dist}\Big({\sf Op}^R_\Xi(g),{\sf sch}\big(\Xi\!\ltimes\!\ell^{\infty,\O}(\Xi)\big)\Big)\ge {\sf D}^\O\big(g\circ\mu^{-1})=\limsup_{(\xi,x)\to\O\times\X}|g(\xi,x)|
\end{equation*}
for every $g\in{\sf VO}^\O(\X\!\times\!\Xi)\circ\mu$\,, a Gohberg type formula for right pseudodifferential operators on the discrete Abelian group $\Xi$\,. The standard case is 
\begin{equation*}\label{gunyali}
{\sf sch}\big(\Xi\!\ltimes\!\ell^{\infty,\delta\Xi}(\Xi)\big)={\sf sch}\big(\Xi\!\ltimes\!c_0(\Xi)\big)=\mathbb K\big[\ell^2(\Xi)\big]\,.
\end{equation*}}
\end{Remark}

\begin{Remark}\label{ayubid}
{\rm One has $\mathbb D^{\delta\Xi}=\mathbb K(\H)$\,, since ${\sf sch}\big(\Xi\!\ltimes\! c_0(\Xi)\big)=\mathbb K\big[\ell^2(\Xi)\big]$ by \cite{Wi} (a well-known fact in the $C^*$-algebra community). This may be considered "the standard case".  The neighborhoods of $\delta\Xi$ in $\beta\Xi$ are the complements of finite subsets of $\Xi$\,, so $\limsup_{\xi\to\delta\Xi}$ may be written as $\limsup_{\xi\to\infty}$\,. A rather common issue in the literature is when is a pseudodifferential operator ${\sf Op}(f)$ compact. By Lemma \ref{gumustekin}, in our setting this happens precisely when 
$$
\limsup_{\xi\to\infty}\sup_{x\in\X}|f(x,\xi)|=0\,.
$$}
\end{Remark}

\begin{Example}\label{aladin}
{\rm Let us consider the simple standard case $\X=\T^N$ (so $\Xi=\mathbb Z^N$) with $\O=\delta\mathbb Z^N$\!, leading above to the important ideal of compact operators. We have $\ell^{\infty,\delta\mathbb Z^N}\!\big(\mathbb Z^N\big)=c_0(\mathbb Z^N)$\,. To show that ${\sf vo}^{\delta\mathbb Z^N}\!\big(\mathbb Z^N\big)$ is much larger than its ideal $c_0\big(\mathbb Z^N\big)$\,, consider first a function 
$$
\dot\psi=\dot\psi_0+\dot\psi_1:\mathbb R^N\!\to\mathbb C\,,
$$
where $\dot\psi_0$ is homogeneous of degree $0$ and $\dot\psi_1\in C_0\big(\mathbb R^N\big)$\,. It is easy to check that the restriction $\psi$ of $\dot\psi$ to $\mathbb Z^N$ is in ${\sf vo}^{\delta\mathbb Z^N}\!\big(\mathbb Z^N\big)$\,. The component $\dot\psi_0$ (admitting radial limits at infinity in $\R^N$) already exhibits the difference. More generally, suppose that $\psi$ is the restriction to $\mathbb Z^N$ of a $C^1$-function $\dot\psi$ on $\R^N$ for which all the partial derivatives tend to zero at infinity. The formula
$$
\dot\psi(\alpha+\beta)-\dot\psi(\alpha)=\int_0^1\!\frac{d}{dt}\dot\psi(\alpha+t\beta)dt=\int_0^1\!\beta\!\cdot\!\nabla\dot\psi(\alpha+t\beta)dt\,,\quad\alpha,\beta\in\R^N
$$
shows immediately that $\psi\in{\sf vo}^{\delta\mathbb Z^N}\!\big(\mathbb Z^N\big)$\,. An example of such a function is $\dot\psi(\alpha):=h\big(|\alpha|^p\big)$ (slightly corrected around the origin), where $h$ is a bounded $C^1$-function of a single variable with bounded derivative and $p\in(-\infty,1)$\,. So $\psi(\xi):=\cos\big(\sqrt{|\xi|}\big)$ is vanishing oscillating in our sense, not having the previous form.}
\end{Example}

\begin{Remark}\label{cocan}
{\rm Among the $C^*$-algebras ${\sf VO}^\O(\X\!\times\!\Xi)$\,, the smallest one is ${\sf VO}^{\delta\Xi}(\X\!\times\!\Xi)$\,. A H\"ormander type symbolic calculus has not been developed for all Abelian compact groups, but when this is the case, the class $S^0_{1,0}(\X\!\times\!\Xi)$\,, for which a standard form of the Gohberg Lemma has been proven \cite{DR,MW}, is actually contained in ${\sf VO}^{\delta\Xi}(\X\!\times\!\Xi)$\,. In \cite[Remark\,4.2]{DR}, where the result needs a Lie group structure (maybe non-commutative), it is stated that much less than the full H\"ormander class is needed. Using notations as ours, what is required is that $f$ and $\partial^\beta_xf$ are bounded if $|\beta|=1$\,, and that 
\begin{equation}\label{otoman}
|(\Delta_qf)(x,\xi)|\le C\<\xi\>^{-\rho}\,,\quad\forall\,(x,\xi)\in\X\!\times\!\Xi
\end{equation}
for some $\rho$ strictly positive and for certain suitable functions $q:\X\to\mathbb C$\,, as described in \cite{RT1,RT2}, where the meaning of $\<\xi\>$ is also provided. We do not need any smoothness in $\X$\,, one has ${\sf osc}^\zeta_\psi=\Delta_\zeta$ for each character $\zeta:\X\to\mathbb C$ and the decrease in $\xi$ is arbitrarily slow.
}
\end{Remark}

\begin{Remark}\label{gunduz}
{\rm A compact Abelian Lie group $\X$ is (isomorphic to) a finite extension of a torus $\T^N$. The finite contribution is trivial if and only if $\X$ is connected. So the dual $\widehat\X\equiv \Xi$ is an extension involving $\mathbb Z^N$ and a finite group, this one being trivial precisely when $\Xi$ is torsion-free. We conclude that the class of compact Abelian Lie group is very restrictive compared with the huge class of compact Abelian groups. However, the fact that non-commutative $\X$ are treated  in \cite{DR} is remarkable; in particular special features appear when one tries to figure out relevant replacements for the absolute value $|f(x,\xi)|$ when this one is an operator; see Theorems 3.1 and 3.2 in \cite{DR}. Our approach fails in the non-commutative case, mainly because the unitary dual of $\X$ is no longer a group. Generalizing the result in \cite{Gr} (with previous contributions \cite{Go,Ho,KN}), valid for $\X=\R^N$, to non-compact Abelian groups also needs some new ideas; this is under investigation.
}
\end{Remark}


\bigskip
{\bf Acknowledgements.} The author has been supported by the Fondecyt Project 1200884. 

\bigskip
ADDRESS

\smallskip
M. M\u antoiu:

Facultad de Ciencias, Departamento de Matem\'aticas, Universidad de Chile 

Las Palmeras 3425, Casilla 653, Santiago, Chile.

E-mail: {\it mantoiu@uchile.cl}

\end{document}